\documentclass[11pt]{amsart}

\usepackage[margin=1.25in]{geometry}
\usepackage{multicol}
\usepackage{hyperref}
\usepackage{amsmath}
\usepackage{amscd}
\usepackage{amssymb}
\usepackage{amsthm}
\usepackage{array}
\usepackage{graphicx}
\usepackage{xy}     
\usepackage{comment}
\usepackage{fancyhdr}
\usepackage{tikz}
\usepackage{tikz-cd}

\theoremstyle{plain}
\newtheorem{theorem}{Theorem}[section]
\newtheorem{lemma}[theorem]{Lemma}
\newtheorem{proposition}[theorem]{Proposition}

\theoremstyle{definition}

\newcommand{\Out}{\mathrm{Out}}
\renewcommand{\S}{\mathcal{S}}
\newcommand{\FF}{\mathcal{F}}

\newcommand{\C}{\mathcal{C}}
\newcommand{\X}{\mathcal{X}}

\begin{document}

\title{A note on subfactor projections}
\author[S. Taylor]{Samuel J. Taylor}
\address{Department of Mathematics, 
University of Texas at Austin, 
1 University Station C1200, 
Austin, TX 78712, U.S.A.
}
\email{\href{mailto:staylor@math.utexas.edu}{staylor@math.utexas.edu}}
\date{June 2013}
\thanks{The author is partially supported by NSF RTG grants DMS-0636557 and DMS-1148490}

\begin{abstract}
We extend some results of \cite{BFproj} on subfactor projections to show that the projection of a free factor  $B$ to the free factor complex of the free factor $A$ is well-defined with uniformly bound diameter, unless either $A$ is contained in $B$ or $A$ and $B$ are vertex stabilizers of a single splitting of $F_n$, i.e. they are disjoint. These projections are shown to satisfy properties analogous to subsurface projections, and we give as an application a construction of fully irreducible outer automorphisms using the Bounded Geodesic Image Theorem.
\end{abstract}
\maketitle

\section{Introduction} \label{intro}
In their recent work on the geometry of $\Out(F_n)$, Mladen Bestvina and Mark Feighn define the projection of a free factor $B < F_n$ to the free splitting complex (or free factor complex) of the free factor $A$, when the two factors are in ``general position.'' They show that these subfactor projections have properties that are analogous to subsurface projections used to study mapping class groups, and they use their results to show that $\Out(F_n)$ acts on a product of hyperbolic spaces in such a way that exponentially growing automophisms have positive translation length.

Because the authors were primarily interested in projections to the splitting complex of a free factor, relatively strong conditions were necessary in order to guarantee that the projections have uniformly bounded diameter, i.e. that they are well-defined. They show that one may project $B$ to the splitting complex of $A$ if either $A$ and $B$ have distance at least $5$ in the free factor complex of $F_n$ or if they have the same color in a specific finite coloring of the factor complex. In this note, we show that if one considers projections to the \emph{free factor complex of a free factor}, simpler and more natural conditions can be given. In particular, we show that for free factors $A,B < F_n$ with rank$(A) \ge 2$ the projection $\pi_A(B) \subset \FF(A)$ into the free factor complex of $A$ is well-defined so long as $(1)$ $A$ is not contained in $B$, up to conjugation, and $(2)$ $A$ and $B$ are not disjoint. This exactly mimics the case for subsurface projection. Here, free factors $A$ and $B$ are \emph{disjoint} if they are distinct vertex stabilizers of a splitting of $F_n$, or equivalently, if they can be represented by disjoint subgraphs of a marked graph $G$. These are also the obvious necessary condition for the projection to be defined. As a consequence of this more inclusive projection, we are able to merge the Bestvina-Feighn projections with those considered in \cite{Tayl1}. 

The first part of this note should be considered as a direct follow-up to the work of Bestvina and Feighn, as our arguments rely heavily on the techniques developed in \cite{BFproj}. Our contribution toward defining subfactor projections is an extension of their results. In summary, we show:

\begin{theorem} \label{main}
There is a constant $D$ depending only on $n = \mathrm{rank}(F_n)$ so that if $A$ and $B$ are free factors of $F_n$ with $\mathrm{rank}(A) \ge 2$, then either 
\begin{enumerate}
\item $A \subset B$, up to conjugation, 
\item $A$ and $B$ are disjoint, or 
\item $\pi_A(B) \subset \FF(A)$ is defined and has diameter $\le D$.
\end{enumerate}
 Moreover, these projections are equivariant with respect to the action of $\Out(F_n)$ on conjugacy classes of free factors and they satisfy the following: There is an $M\ge 0$ so that if free factors $A,B < F_n$ overlap and $G$ is a marked $F_n$-graph, then
$$\min\{ d_A(B,G), d_B(A,G) \} \le M. $$
\end{theorem}

Here, free factors \emph{overlap} if one is not contained in the other, up to conjugation, and they are not disjoint. Hence, for overlapping free factors both subfactor projections are defined. For subsurface projections, the final property in Theorem \ref{main} is known as Behrstock's inequality \cite{Be}. We also have the following strengthening of the Bounded Geodesic Image Theorem of \cite{BFproj}. For subsurface projections, this was first shown in \cite{MM2}.

\begin{theorem}
For $n\ge 3$, there is $M \ge 0$ so that if $A$ is a free factor of $F_n$ with $\mathrm{rank}(A) \ge 2$ and $\gamma$ is a geodesic of $\FF_n$ with each vertex of $\gamma$ meeting $A$ (i.e. having well-defined projection to $\FF(A)$) then $\mathrm{diam}(\pi_A(\gamma)) \le M$.
\end{theorem}

Finally, as an application of subfactor projections we give a construction of fully irreducible automorphism similar to Proposition $3.3$  of \cite{Ma2}, where pseudo-Anosov mapping classes are constructed. Here, free factors $A$ and $B$ \emph{fill} $F_n$ if no free factor $C$ is disjoint from both $A$ and $B$.

\begin{theorem} \label{introConstruct}
Let $A$ and $B$ be rank $\ge 2$ free factors of $F_n$ that fill and let $f,g \in \Out(F_n)$ satisfy the following:
\begin{enumerate}
\item $f(A) = A$ and $f|_A \in \Out(A)$ is fully irreducible, and
\item $g(B) =B$ and $g|_B \in \Out(B)$ is fully irreducible.
\end{enumerate}
Then there is an $N \ge 0$ so that any nontrivial automorphism in the subgroup $\langle f^N,g^N \rangle \le \Out(F_n)$ that is not conjugate to a power of $f$ or $g$ is fully irreducible. 
\end{theorem}

\noindent See Section \ref{constructing} for a stronger statement. Theorem \ref{introConstruct} adds a new construction of fully irreducible automorphisms to the methods found in \cite{CPtwisting}, where they arise as compositions of Dehn twists, and in \cite{KLping}, where they are compositions of powers of other fully irreducible automorphisms.

As a final remark, we warn the reader that the projection $\pi_A(\cdot)$ is into the free factor complex of $A$ and $d_A(\cdot,\cdot)$ denotes distance in $\FF(A)$. This is different from \cite{BFproj} where these symbols denote projections and distance in the free \emph{splitting} complex of $A$, denoted $\S(A)$. Because of the simple conditions under which subfactor projections into the free factor complex are defined, we hope that this note convinces the reader that projecting to the factor complex of a free factor is a useful notion of projection. An entirely different type of projection for free groups appears in \cite{SS}, and the relationship between these projections is explained in \cite{Tayl1}. \\

\noindent \textbf{Acknowledgements.}  Many thanks are due to Mladen Bestvina for his insight and helpful conversation, as well as his encouragement to write up this note. The author also thanks Alan Reid for his advice and support and Hossein Namazi for suggesting the use of the graph $\C_n$ in Section \ref{constructing}.

\section{Background}\label{background}
We briefly review some background material needed for this note and refer the reader to the references below for additional details. Denote by $F_n$ the free group of rank $n$ and by $\Out(F_n)$ its group of outer automorphisms. A \emph{graph} is a $1$-dimensional CW complex and a \emph{tree} is a simply connected graph. A finite graph is a \emph{core} graph if all its vertices have valence at least $2$ and any connected graph with finitely generated, nontrivial fundamental group has a unique core subgraph that carries its fundamental group. A core graph has a unique CW structure, or \emph{triangulation}, where each vertex has valence least $3$ and we refer to vertices and edges in this triangulation as \emph{natural}. If the modifier natural is omitted then we are referring to the graph with its given triangulation.

By a \emph{free splitting} of $F_n$, we mean a minimal action of $F_n$ on a nontrivial simplicial tree $T$ with trivial edge stabilizers. Recall that the action $F_n \curvearrowright T$ is \emph{minimal} if there is no invariant subtree. By Bass-Serre theory, free splittings of $F_n$ correspond to graph of groups decompositions of $F_n$ with trivial edge groups; we will make free use of both of these perspectives.  An equivariant  map $T \to T'$ between splittings is a \emph{collapse map} if all point preimages are connected. In this case, we say that $T$ \emph{refines} $T'$. The splittings $T,T'$ are \emph{conjugate} if there is a equivariant homeomorphism $T \to T'$. 

The \emph{free splitting complex} of $F_n$, denoted $\S_n$, is the complex whose vertices are conjugacy classes of $1$-edge free splittings of $F_n$ and two vertices are joined by an edge if they have a common refinement. See \cite{HMsplit} for details. The \emph{free factor complex} of $F_n$, denoted $\FF_n$, for $n\ge 3$ is the complex whose vertices are conjugacy classes of free factors and factors $A_0, \ldots, A_n$ span an $n$-simplex if after choosing representatives and possibly reordering $A_0 \subset  \ldots \subset A_n$. $\FF_n$ was introduced in \cite{HVff}. When $n =2$, $\FF_2$ is modified to be the graph whose vertices are conjugacy classes of rank $1$ free factors and two vertices are joined by an edge if there are representatives of each that together form a basis for $F_2$. This makes $\FF_2$ into the standard Farey graph. Remark that throughout this note, we sometimes blur the distinction between a free factor and its conjugacy class when it is clear from context what is meant.

Both $\FF_n$ and $\S_n$ are known to be hyperbolic. This was first show for $\FF_n$ in \cite{BFhyp} and for $\S_n$ in \cite{HMsplit}. See also \cite{KapRaf,HHhyp}. Relating these complexes, there is a coarse $4$-Lipschitz map $\pi: \S_n \to \FF_n$ given by mapping the splitting $T$ to its vertex stabilizers in $\FF_n$. For an arbitrary free splitting $T$ of $F_n$, we use the same notation to denote the map that associates to $T$ the set of free factors that arise as a vertex stabilizer of a one-edge collapse of the $T$.

To study $\Out(F_n)$, Culler and Vogtmann introduces Outer space $\X_n$, the space of metric graphs marked by $F_n,$ or equivalently, the space of minimal, proper actions of $F_n$ on simplicial $\mathbb{R}$-trees \cite{CVouter}. Recall that a \emph{marking} of the graph $G$ is a homotopy equivalence $\phi: R_n \to G$, where $R_n$ is the rose with $n$ petals whose fundamental group has been identified with $F_n$. A metric $l: E(G) \to \mathbb{R}_{+}$ on the marked graph $G$ is an assignment of a positive real number, or length, to each edge of $G$ and a marked metric graph is the ordered triple $(G,\phi,l)$, which we usually simplify to $G$. The volume of $G$ is the sum of the lengths of the edges of $G$. \emph{Outer space} $\X_n$ is defined to be the space of marked metric core graphs of volume one, up to equivalence. Here, $(G,\phi,l)$ and $(G',\phi',l')$ are equivalent if there is an isometry $i: G \to G'$ that is homotopic to $\phi' \circ \phi^{-1} : G \to G'$. In general, any map $h: G \to G'$ homotopic to $\phi' \circ \phi^{-1}$ is called a  \emph{change of marking}. For $G \in \X_n$ and $\alpha$ a conjugacy class of $F_n$, let $l_G(\alpha)$ denote the length of the immersed loop in $G$ that correspond to $\alpha$ through the marking for $G$. We use the notation $\hat{\X_n}$ to denote \emph{unprojectivized Outer space}, where the requirement that graphs have volume one is dropped. 

We consider $\X_n$ with its Lipschitz metric defined by
$$d_{\X}(G,G') = \inf \{ \log  L(h): h \simeq \phi' \circ \phi^{-1}\}, $$
where $L(h)$ is the Lipschitz constant for the change of marking $h$ and $\phi: R_n \to G$ and $\phi':R_n \to G'$ are the corresponding markings.  We remark that this (asymmetric) metric induces the standard topology on $\X_n$ that is got by considering lengths of immersed loops representing conjugacy classes in $F_n$ \cite{FMout}. Also, viewing $\X_n$ as the space of minimal, proper $F_n$-actions on simplicial $\mathbb{R}$-trees, we have the map $\pi: \X_n \to \S_n \to \FF_n$, as described above. Note that free factors in the image $\pi(G)$ of $G \in \X_n$ are represented by embedded subgraphs of $G$.

It is well known that the infimum in the definition of the Lipschitz metric is realized by some (non-unique) \emph{optimal} map \cite{FMout, BFhyp}. We briefly describe the folding path induced by an optimal $f:G \to G'$ and refer to \cite{BFhyp} for more details. First, a \emph{illegal turn structure} on $G$ is an equivalence relation on the set of directions at each vertex of $G$; the equivalence classes are called \emph{gates}. Here, a \emph{turn} is a unordered pair of distinct directions at a vertex and a turn is \emph{illegal} if both directions are contained in the same gate and is \emph{legal} otherwise. An illegal turn structure is a \emph{train track structure} if, in addition, every vertex has at least $2$ gates. For marked graphs $G,G' \in \X_n$ any piecewise linear change of marking map $h: G \to G'$ induces an illegal turn structure on $G$ whose gates are the directions at each vertex that are identified by $h$. In fact, there is always a change of marking $f:G \to G'$, called an \emph{optimal map}, that is constant slope (i.e. stretch) on each edge of $G$ with the property that the subgraph $\triangle(f) \subset G$ consisting of edges of maximal slope, $L(f)$, is a core subgraph and that the illegal turn structure on $G$ induced by $f$ restricts to a train track structure on $\triangle(f)$ \cite{FMout, Bbers}. From these properties, it follows that $f$ has minimal Lipschitz constant over all change of markings $G \to G'$. 
If, in addition, $\triangle(f) = G$, i.e. every edge is stretched by $L(f)$, then there is an induced \emph{folding path} $t \mapsto G_t$ joining $G$ and $G'$ in $\X_n$. Such a path is locally obtained by folding all illegal turns at unit speed and then rescaling to maintain volume one. For each $a \le b$, there is an induced optimal map $f_{ab}:G_a \to G_b$. These folding maps compose naturally and send legal segments to legal segments, where a \emph{legal segment} of $G_a$ is an immersed path that makes only legal turns.  See \cite{BFhyp} for a detailed construction. Arbitrary points $G,G' \in \X_n$ are joined by a geodesic path that first rescales edge lengths of $G$ and is then followed by a folding path. For a folding path $G_t$ in $\X_n$, a family of subgraphs $H_t \subset G_t$ is called \emph{forward invariant} if for all $a \le b$, $H_a$ maps into $H_b$ under the folding map $f_{ab}:G_a \to G_b$.

Finally, we recall the projection a splitting of $F_n$ to the free factor complex of a subfactor. See \cite{Tayl1} for details. For $G \in \X_n$ and a rank $\ge 2$ free factor $A$ we can consider the core subgraph of the cover of $G$ corresponding to the conjugacy class of $A$. We denote this marked $A$-graph by $A|G$ and the associated immersion by $p: A|G \to G$. Pulling back the metric on $G$, we obtain $A|G \in \hat{\X}(A)$. Denote by $\pi_A(G) = \pi(A|G) \subset \FF(A)$ the projection of $A|G$ to the free factor complex of $A$. Alternatively, if $G$ corresponds to the action $F_n \curvearrowright T$ (i.e. $T$ is the universal cover of $G$) with \emph{minimal $A$-subtree} $T^A$, then $A \curvearrowright T^A$ represents a point in $\hat{\X}(A)$. The projection $\pi_A(T) = \pi_A(G) \subset \FF(A)$ is the set of free factors of $A$ that arise as vertex stabilizers of one-edge collapses of $T^A$. Note that this projection is defined whenever $T$ is a splitting of $F_n$ where $A$ does not fix a vertex (i.e. where $T^A$ is not trivial).

For a free factor $A$ of $F_n$, we use the symbol $d_A$ to denote distance in $\mathcal{F}(A)$ and for $F_n$-trees $T_1,T_2$ we use the shorthand 
$$d_A(T_1,T_2) := d_A (\pi_A (T_1), \pi_A(T_2)) = \text{diam}_A(\pi_A (T_1) \cup \pi_A(T_2)),$$
when both projections are defined. 

\section{Folding paths and the Bestvina-Feighn projections} \label{BFfolding}
Let $A$ and $B$ be (conjugacy classes of) free factors of $F_n$ with $\mathrm{rank}(A) \ge 2$. Suppose that $A$ and $B$ are not disjoint and that $A$ is not contained in $B$, up to conjugation. In this case, we say that $B$ \emph{meets} $A$. Define \emph{the projection of $B$ to the free factor complex of $A$} to be the following subset of $\FF(A)$:
\begin{eqnarray*}
\pi_A(B) &=& \bigcup \{\pi_A(T): T \text{ is a splitting of } F_n \text{ with vertex stabilizer } B  \} \\
              &=&  \bigcup \{\pi_A(G): G\in \X_n \text{ and } B|G \subset G \text{ is embedded}  \}. 
\end{eqnarray*}
In other words, $\pi_A(B)$ is the set of vertex groups of splittings of $A$ that are refined by the splitting $A \curvearrowright T^A$, where $T$ is any free splitting with vertex stabilizer $B$. For convenience, if $A \subset B$ or $A$ and $B$ are disjoint we define $\pi_A(B)$ to be empty and say that $B$ \emph{misses} $A$. If $A$ meets $B$ and $B$ meets $A$, then  both projections are nonempty and we say that $A$ and $B$ \emph{overlap}. Note that the conditions for $\pi_A(B)$ to be nonempty are precisely that the tree $T^A$ is non-degenerate for any choice of $T$ with $B$ as a vertex stabilizer. The main result of this note is that $\mathrm{diam}(\pi_A(B))$ is uniformly bounded and, therefore, can be used as a coarse projection. This is shown in \cite{BFproj} in the case that either $d_{\FF}(A,B) >4$ or $A$ and $B$ have the same color in a specific finite coloring of $\FF_n$. This, however, excludes cases of interest; for example when the free factors have nontrivial intersection, as in \cite{Tayl1}. We note that by ``uniformly bounded'' we mean bounded by a constant depending only on $n$, the rank of $F_n$. Unlike the subsurface case, where the bound is $3$, we do not explicitly compute this constant.

Much of the machinery needed to prove Theorem \ref{main} has already been obtained in \cite{BFproj}. We recall some of the technical results from that paper that are needed here. Suppose that $G_t$ is a folding path for $t\in [\alpha,\omega]$ as in Section \ref{background} and that $A$ is a free factor. Then for all $t \in [\alpha, \omega]$, we have the immersion $p_t: A|G_t \to G_t$ corresponding to the core of the $A$-cover of $G_t$ and $A|G_t$ induces a path in $\hat{\X}(A)$. The results of \cite{BFproj} explain the behavior of the path $A|G_t$ and track the progress of $\pi_A(G_t) = \pi(A|G_t)$ in $\FF(A)$. Note that $p_t: A|G_t \to G_t$ induces an illegal turn structure on $A|G_t$. Call a valence $2$ vertex, i.e. a vertex appearing in the interior of a natural edge, an \emph{interior illegal turn} if it has only one gate.

\begin{lemma}[Lemma 3.1 of \cite{BFproj}] \label{lemma3.1}
For a folding path $G_t$, $t\in [\alpha , \omega]$ and a finitely generated subgroup $A < F_n$, the interval $[\alpha, \omega]$ can be divide into three subintervals $[\alpha, \beta)$, $[\beta, \gamma)$, and $[\gamma,\omega]$ so that the following properties characterize the restriction of $A|G_t$ to the middle interval $[\beta, \gamma)$: all vertices of $A|G_t$ have $\ge 2$ gates, there are no interior illegal turns, and all natural edges of $A|G_t$ have length $< 2$.  Moreover, the images of $\{A|G_t: t\in [\alpha, \beta)\}$ and $\{ A|G_t: t\in [\gamma, \omega]\}$ in $\S(A)$ (and $\FF(A)$) have uniformly bounded diameter.
\end{lemma}

From this lemma, it is shown that the projection of the folding path $G_t$ to the free splitting (or free factor) complex of $A$ is an unparameterized quasi-geodesic with uniform constants. We will not need this fact in what follows.  Note that for $a,b \in [\beta ,\gamma)$, where $[\beta,\gamma)$ is the middle interval given in Lemma \ref{lemma3.1}, the folding map $f_{ab}: G_a \to G_b$ induces a map $A|G_a \to A|G_b$ between the cores of the $A$-covers.

For the immersion $p: A|G \to G$, define $\Omega \subset G$ as the set of edge of $G$ that are at least double covered by $p$ and set $\tilde{\Omega} \subset A|G$ to be the subgraph $p^{-1}(\Omega) \subset A|G$. If $\tilde{\Omega} = \emptyset$, then $A|G \to G$ is an embedding and we say that $A$ (or $A|G$) is \emph{embedded} in $G$. If $\tilde{\Omega}$ is a forest (a disjoint union of trees), then we say that $A$ (or $A|G$) is \emph{nearly embedded}. The following lemma states that if a folding path makes significant progress in $\FF(A)$ then $A$ must be nearly embedded along the path.

\begin{lemma}\label{forest}
Let $G_t$ be a folding path for $t \in [\alpha,\omega]$ and let $[\beta, \gamma)$ be the middle interval determined by Lemma \ref{lemma3.1}. Then after restricting $G_t$ to $t\in [\beta, \gamma)$, the subgraph $\tilde{\Omega}_t \subset A|G_t$ is forward invariant and if for some $t_0$, $\tilde{\Omega}_{t_0}$ is not a forest (i.e. $A$ is not nearly embedded in $G_{t_0}$), then $\pi_A(\{G_t : t\ge t_0 \})$ has uniformly bounded diameter in $\FF(A)$.
\end{lemma}

\begin{proof}
That $\tilde{\Omega}_t$ is forward invariant on the middle interval is contained in Lemma 4.3 of \cite{BFproj}. The other statement is essentially Lemma 4.4 in \cite{BFproj}. There, it is shown that if $\tilde{\Omega}_{t_1} = A|G$ then $\pi_A(\{G_t : t\ge t_1 \})$ is uniformly bounded. Since $\tilde{\Omega}_{t}$ is forward invariant it suffices to show that progress of $A|G_t$ in $\FF(A)$ is bounded so long as $\tilde{\Omega}_t$ is a proper subgraph of $A|G_t$ that is not a forest. Suppose this is the case for $\tilde{\Omega}_{t_0} \subset A|G_{t_0}$ and let $x_0$ be an immersed loop in $A|G_{t_0}$ that is contained in $\tilde{\Omega}_{t_0}$. Denote by $x_t$ the immersed representative of the image of $x_0$ through $A|G_{t_0} \to A|G_{t}$. Since $\tilde{\Omega}_t$ is a proper subgraph, $x_t$ fails to cross some edge of $A|G_t$. This implies that the cyclic free factor represented by $x_0$ has distance $\le 5$ from $\pi_A(G_t) = \pi (A|G_t)$ in $\FF_n$, so long as $\tilde{\Omega}_t$ is a proper subgraph. This completes the proof.
\end{proof}

\section{Diameter bounds} \label{diameter}
The following lemmas determine when the projection of a factor $B$ to the free factor complex of the factor $A$ is well-defined. The first provides a criterion for when two free factors can be embedded in a common marked graph and the second shows that the failure of a joint embedding is enough to block progress of subfactor projections along a folding path. We recall that in \cite{BFproj}, the authors show that if the finitely generated subgroup $A < F_n$ is nearly embedded in $G$, then $A$ is a free factor of $F_n$. Similar arguments are used to prove the following:

\begin{lemma}\label{jointlyembedded}
Suppose that $p: A|G \to G$ is the canonical immersion and that $B|G \subset G$ is an embedding for free factors $A$ and $B$ of $F_n$. Let $E^B$ be the collection of edges of $A|G$ that map to edges of $B|G$. If $\tilde{\Omega} \cup E^B \subset A|G$ is a forest, then there is a marked graph $G'$ where $A$ and $B$ are disjointly embedded.
\end{lemma}

\begin{proof}
Enlarge the forest $\tilde{\Omega} \cup E^B$ to a maximal tree $T$ and let $E$ be the set of edges not contained in $T$. These edges are in bijective correspondence with the edges of $p(E)$, since they are not in $\tilde{\Omega}$. For $x \in T$, define 
$$G' = A|G \vee_{x=p(x)} (G \setminus p(E)). $$
As in \cite{BFproj}, we have the morphism (edge isometry) $G' \to G$ induced by $p: A|G \to G$ and the inclusion of $G \setminus p(E)$ into $G$. Folding the edges of the $T$ into $G \setminus p(E)$, we arrive at an intermediate graph $G''$ with an induced morphism $G'' \to G$. Because $T$ is a tree, such folds do not change the homotopy type of the graph. Further, since no edges outside of $T$ are identified when mapped to $G$, the morphism $G'' \to G$ is bijective. We conclude that the map $G' \to G$ is a homotopy equivalence and that $G'$ contains disjoint, embedded copies of both $A|G' = A|G$ and $B|G' =B|G \subset G \setminus p(E)$.

\end{proof}

We show that for any marked graphs $G$ and $G'$ where $B$ is embedded, $d_A(G,G')$ is uniformly bounded. For this, fix a marked graph $G_0$ that is a rose and for which $B$ is embedded. For any metric graph $G \in \mathcal{X}_n$ with $B|G$ embedded we can choose edge lengths for $G_0$ so that $G_0\in \mathcal{X}_n$ and there is an optimal map $f: G_0 \to G$ with $\triangle(f) = G_0$ and $f(B|G_0) \subset B|G$. Then the folding path $\{G_t: t \in [0,T]\}$ induced by $f$ with $G_T =G$ has the property that $B|G_t$ is embedded in $G_t$ for all $t \in [0,T]$, and for all $s \le t$, $f_{st}:G_s \to G_{t}$ maps $B|G_s$ into $B|G_t$. Hence, $B|G_t$ is forward invariant. It suffices to show that the image  $\pi_A(G_t) \subset \FF(A)$ of the folding path is bounded by a constant depending only on $n$. To do this, first restrict to a subinterval $[a,b] \subset [0,T]$ where
\begin{enumerate}
\item for $t \in [a,b]$, the immersion $p_t: A|G_t \to G_t$ induces a train track structure on $A|G_t$, i.e. $A|G_t$ has no interior illegal turns and all vertices have $\ge 2$ gates. Also, each natural edge of $A|G_t$ has length $< 2$,
\item for $t\in [a,b]$, the subgraph $\tilde{\Omega}_t \subset A|G_t$ is a forward invariant forest, and
\item the projections $\pi_A (\{G_t : t \in [0,a] \})$ and $\pi_A (\{G_t : t \in [b,T] \})$ in $\FF(A)$ have uniformly bounded diameter.
\end{enumerate}

Note the such an interval exists by Lemma \ref{lemma3.1} and Lemma \ref{forest}. For $p:A|G_t \to G_t$, let $E^B_t \subset A|G_t$ be the set of edges in the triangulation induced from $G_t$ that project to edges of $B|G_t \subset G_t$, as in Lemma \ref{jointlyembedded}.

\begin{lemma} \label{progressbound}
With $\{G_t: t \in [a,b]\}$ as above, if there is a $t_0 \in [a,b]$ so that $A|G_{t_0}$ has an embedded loop $x_0$ all of whose edges are contained in $\tilde{\Omega}_{t_0} \cup E^B_{t_0}$, then the projection $\pi_A(\{A|G_t: t\ge t_0\})$ has uniformly bounded diameter in $\FF(A)$.
\end{lemma}

\begin{proof}
Let $x_t$ be the image of $x_0$ in $A|G_t$ pulled tight, i.e. its immersed representative. We show that for any edge $e$ of $A|G_t$ not in $E^B_t$, $x_t$ crosses $e$ a bounded number of times. Since by assumption $A$ is not contained in $B$ such a edge is guaranteed to exist. By \cite{BFhyp}, this implies that $\pi_A(G_t) = \pi(A|G_t)$ has bounded distance from the cyclic factor of $A$ represented by $x_0$, for all $t \ge t_0$.

Suppose that $e$ is an edge of $A|G_t$ not contained in $E^B_t$ and let $p$ be a point in the interior of $e$. Note that $x_0$ is composed of a bounded number of legal segments of $\tilde{\Omega}_{t_0}$ and edges of $E^B_{t_0}$. To see this, recall that since $x_0$ is embedded it consists of a bounded number of natural edges of $A|G_{t_0}$, each of which is legal because $A|G_t$ has no interior illegal turns. Also, the number of edges of $E^B_{t_0}$ not appearing in $\tilde{\Omega}_{t_0}$ is bounded by $3\cdot \mathrm{rank}(B) -3$ since there are no more of these edges than edges of $B|G_{t_0}$. Hence, each natural edge of $A|G_{t_0}$ crossed by $x_0$ is contained in a bounded number of legal segments of $\tilde{\Omega}_{t_0}$ plus edges of $E^B_{t_0}$ that are not contained in $\tilde{\Omega}_{t_0}$.

Let $s$ be a legal segment of $\tilde{\Omega}_{t_0}$ that maps overs $p$ more than twice. Then by forward invariance of $\tilde{\Omega}_t$ and legality of $s$, $p$ must be contained in the core of $\tilde{\Omega}_t$. This contradicts the assumption that $\tilde{\Omega}_t$ is a forest. For an edge of $E^B_{t_0}$ we note that by forward invariance of $B|G_t$, no edge of $B|G_{t_0}$ can map over the edge $p(e)$. Hence, no edge of $E^B_{t_0}$ can map over $e$. We conclude that $x_t$ crosses $e$ no more that $2 \cdot |$legal segments of $\tilde{\Omega}_{t_0}|$ times. Since we have seen that this quantity is bounded by a constant depending only on the rank of $G$, we conclude that $\pi_A(\{G_t: t\ge t_0\})$ is uniformly bounded.
\end{proof}

Together, these lemmas complete the proof of our main theorem.

\begin{theorem}\label{well}
Let $A$ and $B$ be conjugacy classes of free factors of $F_n$ with $\mathrm{rank}(A) \ge 2$. Then either  $A$ and $B$ are disjoint, $A\subset B$, or $\pi_A(B)$ is well-defined with uniformly bounded diameter.
\end{theorem}

\begin{proof}
Suppose that $A$ and $B$ are free factors that are not disjoint and that $A$ is not contained in $B$, up to conjugation. Let $T$ be any free splittings of $F_n$ with $B$ as a vertex stabilizer and take $G \in \mathcal{X}_n$ to be a graph refining the splitting $T$, so that $B|G$ is embedded in $G$. Let $G_0$ be the marked rose discussed above and construct the folding path $\{G_t: t\in [0,T]\}$ from $G_0$ to $G_T=G$ with subinterval $[a,b] \subset [0,T]$ satisfying conditions $(1),(2),$ and $(3)$.

If $d_A(G_a,G_b)$ is larger than the bound determined in Lemma \ref{progressbound}, then $A|G_a$ does not contained an embedded loop with edges in $\tilde{\Omega}_{a} \cup E^B_{a}$, so $\tilde{\Omega}_{a} \cup E^B_{a}$ is a forest. By Lemma \ref{jointlyembedded}, this implies that there is a marked graph where $A$ and $B$ are disjointly embedded, contradiction our assumption. Hence,
$$d_A(G_0,T) \le d_A(G_0,G)+4 \le  d_A(G_0,G_a) + d_A(G_a,G_b) +d_A(G_b,G) +4$$
where the first and third terms are uniformly bounded by condition $(3)$ in the properties of the folding path $G_t$ and the second term is no larger than the bound determined in Lemma \ref{progressbound}. Since $T$ was an arbitrary splitting of $F_n$ with vertex stabilizer $B$, this completes the proof.
\end{proof}

Having shown that subfactor projections are well-defined, we collect some basic facts. First, for the free group $F_n$, let $D$ denote the constant determined in Theorem \ref{well} so that if $B$ meets $A$ then diam$(\pi_A(B)) \le D$. For free factors $A,B$ each of which meet the rank $\ge 2$ free factor $C$ set
$$d_C(A,B) :=d_C(\pi_C(A),\pi_C(B)) = \mathrm{diam}(\pi_C(A) \cup \pi_C(B)), $$
where $d_C$ denotes distance in $\FF(C)$. If, additionally, $A$ and $B$ are adjacent vertices of $\FF_n$ then (up to switching $A$ and $B$) $A \subset B$ and so $d_C(A,B) \le 2D$, since each projection contains the projection of a graph where both $A$ and $B$ are embedded. This shows that the projection to $\FF(C)$ is coarsely Lipschitz along paths in $\FF_n$ all of whose vertices meet $C$. We further remark the projection $\pi_C: \X_n \to \FF(C)$ is coarsely Lipschitz; this follows from facts in \cite{BFproj} and is proven explicitly in the appendix of \cite{Tayl1}. Finally, the following naturality property is a direct consequence of the definitions: if $f \in \Out(F_n)$ and $A$ and $B$ are free factors that meet the rank $\ge 2$ free factor $C$ then 
$$d_{fC}(fA,fB) = d_C(A,B),$$
where we use the natural action of $\Out(F_n)$ on conjugacy classes of free factors.

\section{Properties} \label{properties}
The following properties of subfactor projection are obtained just as in \cite{BFproj}. The point here is that our conclusions hold for more general pairs of free factors, so long as we project into the free factor complex rather than free splitting complex of a free factor. Some proofs are provided for completeness and as a verification that they apply in our more general setting. We first have the following version of Lemma $4.12$ of \cite{BFproj}.
 
 \begin{lemma} \label{nearlyembedded}
Suppose that $A$ is nearly embedded in $G \in \X_n$. Then there is a $G' \in \X_n$ where $A$ is embedded and a path in $\X_n$ from $G$ to $G'$ with the property that for any free factor $B$ which $A$ meets, the projection of this path to $\FF(B)$ has uniformly bounded diameter.
 \end{lemma}
 
 \begin{proof}
 We refer to the proof of Lemma \ref{jointlyembedded}. Since $A$ is nearly embedded in $G$, $\tilde{\Omega} \subset A|G$ is a forest. Let $T$ be a maximal tree containing $\tilde{\Omega}$ and set $E$ to be the set of edge of $A|G$ not contained in $T$. Recall that $p: A|G \to G$ maps edges of $E$ bijectively to edges of $p(E)$.  If the image of $B|G$ in $G$ crosses no edge of $p(E)$ then $B$ is carried by the subgraph $G \setminus p(E)$ of $G' = A|G \vee_{x=p(x)} (G \setminus p(E))$. This contradicts our assumption that $A$ meets $B$. Hence, the image of $B|G$ crosses the image of some edge $e$ of $E$ in $A|G$. In the language of \cite{BFproj}, $B$ is \emph{good} for $A$. The required path $ G_t$ from $G'$ to $G$ is then the path determined by folding the morphism $G' \to G$ given in Lemma \ref{jointlyembedded}. This path makes only bounded progress in $\FF(B)$, indeed in $\S(B)$, as shown in \cite{BFproj}. The point is that the splitting of $B$ determined by the preimage of the midpoint of $p(e)$ through the map $B|G_t \to G_t$ is unaltered along the path.
 \end{proof}

\begin{theorem} \label{Behrstock}
Given $F_n$, there is an $M \ge 0$ so that if $A$ and $B$ are overlapping free factors of rank $\ge 2$ then for any splitting $T$ that meets both factors
$$\min\{d_A(B,T), d_B(A,T)  \} \le M. $$
\end{theorem}

\begin{proof}
We follow the proof of Proposition $4.14$ of \cite{BFproj} and use Lemma \ref{nearlyembedded} above. Assume that both $d_A(B,T)$ and $d_B(A,T)$ are very large (relative to $D$) and let $G \in \X_n$ be a refinement of $T$ (as a splitting of $F_n$). Define a folding path $G_t$, $t\in [0,S]$ from $G_0$ to $G_S =G$, where $G_0$ is any graph with $A$ embedded. Since $d_B(A,T)$ and, hence, $d_B(G_0,G_S)$ is large, by Lemma \ref{forest} there is a subinterval $[t_1,t_2]$ where $B$ is nearly embedded and where $G_t$ makes large progress in $\FF(B)$, i.e. $d_B(G_{t_1},G_{t_2})$ is large.

Since $B$ is nearly embedded in $G_{t_2}$ and $d_A(B,G)$ is big by assumption, Lemma \ref{nearlyembedded} and Lemma \ref{forest} imply that there is an subinterval $[t_3,t_4] \subset [t_2,S]$, where $A$ is nearly embedded. Hence, $G_{t_3}$ is a graph where $A$ is nearly embedded and has very large distance in $\FF(B)$ from $G_0$, where $A$ is embedded. This contradicts Lemma \ref{nearlyembedded} and the fact that diam$(\pi_B(A)) \le D$.

\end{proof}
 
Finally, we note the following version of the Bounded Geodesic Image Theorem. The proof in \cite{BFproj} follows through without change after using the more general conditions for projection that are explained in this note. 

 \begin{theorem}[Bounded Geodesic Image Theorem] \label{BGIT}
For $n\ge 3$, there is $M \ge 0$ so that if $A$ is a free factor of $F_n$ of rank $\ge 2$ and $\gamma$ is a geodesic of $\FF_n$ with each vertex of $\gamma$ meeting $A$ (i.e having nontrivial projection to $\FF(A)$) then $\mathrm{diam}(\pi_A(\gamma)) \le M$.
 \end{theorem}
 
We conclude this section with a remark: Using Theorem \ref{main} one can give a coarse lower bounded on distance in $\Out(F_n)$ or $\X_n$ exactly as in \cite{Tayl1}. Since these lower bounds do not cover all distance in $\Out(F_n)$, i.e. they do not give upper bounds, we do not provide the details here. However, we do note that similar to \cite{BFproj} one needs to bound the size of a collection of rank $\ge 2$ free factors where pairwise projections are not defined. This is done in \cite{BFproj} by finding a finite coloring of the free factor complex so that between similarly colors factors one may project one of the factors to the \emph{splitting complex} of the other. As is a theme of this paper, if we consider projections to factor complexes things become simpler. In particular, if factors $A$ and $B$ are represented by embedded subgraphs in a graph $G$ and each factor represents the same subgroup of $H_1(F_n; \mathbb{Z}/2)$, then these subgraphs must be equal and so $A = B$. Hence, we can provide the following coloring of $\FF_n^0$: define $\mathcal{H}$ to be the set of proper subgroups of $H_1(F_n; \mathbb{Z}/2)$ and let $c: F_n^0 \to \mathcal{H}$ be defined by
$$c(A)  = H_1(A; \mathbb{Z}/2) \le H_1(F_n; \mathbb{Z}/2). $$
Then, as explained above, if $A$ and $B$ are distinct free factors with rank $\ge 2$ and $c(A) =c(B)$, then $A$ and $B$ overlap.

\section{Constructing Fully Irreducible Automorphisms}\label{constructing}

We consider the following modification of the free factor complex. Let $\C_n$ for $n \ge 2$ be the graph defined as follows: the vertices of $\C_n$ are conjugacy classes of rank $1$ free factors of $F_n$ and two vertices $v,w \in \C^0_n$ are jointed by an edge if they can be represented by elements $x$ and $y$ in $F_n$, respectively, such that $\langle x,y \rangle$ is a rank $2$ free factor of $F_n$; that is edges are determined by disjointness of vertices. This graph is obviously quasi-isometric to the free factor complex of $F_n$. For a free factor $A < F_n$, let $X_A$ denote the set of vertices of $\C_n$ that fail to project to $\FF(A)$, i.e. that are disjoint from $A$. The complex $\C_n$ has the following advantage over $\FF_n$: for any free factor $A$ the diameter of $X_A$ in $\C_n$ is $\le 2$. In fact, $X_A$ is contained in a $1$-neighborhood of any rank $1$ free factor of $A$. We remark that for $\gamma_1,\gamma_2$ adjacent vertices of $\C_n$ that are not contained in $X_A$, $d_A(\gamma_1, \gamma_2) \le 2D$. We also have the corresponding version of the Bounded Geodesic Image Theorem for $\C_n$. We state it here for later reference.

\begin{proposition}\label{BGITc}
For $n \ge 3$, there is an $M \ge 0$ so that if $A$ is a free factor of $F_n$ of rank $\ge2$ and $\gamma$ is a geodesic in $\C_n$ with each vertex of $\gamma$ meeting $A$, i.e. $\gamma$ is disjoint from $X_A$, then $\mathrm{diam}(\pi_A(\gamma)) \le M$.
\end{proposition}

To make effective use of the graph $\C_n$, we need the following lemma.

\begin{lemma}\label{cyclicsuffices}
Let $A$ and $B$ be free factors of $F_n$ with $\mathrm{rank}(A) \ge 2$ and $\pi_A(B) \neq \emptyset$. Then there is a cyclic (i.e. rank 1) factor $\gamma \subset B$ with $\pi_A(\gamma) \neq \emptyset$.
\end{lemma}

\begin{proof}
If $B$ is rank $1$ there is nothing to show, and if $B \subset A$ then any rank $1$ subfactor will do. Hence, we may assume that rank$(B) \ge 2$ and that $\pi_B(A) \neq \emptyset$. Choose a cyclic factor $\gamma$ of $B$ that is at distance $> D+4$ from $\pi_B(A)$ in $\FF(B)$.  If $\pi_A(\gamma) = \emptyset$ then there is a marked graph $G$ containing subgraphs representing $A$ and $\gamma$, respectively. Then by definition
$$\pi_B(G) \subset \pi_B(A) \quad \text{ and } \quad \pi_B(G) \subset \pi_B(\gamma), $$
implying that $d_B(A,\gamma) \le \mathrm{diam}(\pi_B(A)) + \mathrm{diam}(\pi_B(\gamma)) \le D+4$, a contradition.

\end{proof}

The following proposition shows how subfactor projections can be used to build up distance in the graph $\C_n$. In the mapping class group situation, this is proven for the curve complex in \cite{Ma2}. The idea originates in \cite{KLconv}.

\begin{proposition} \label{progress}
Let $\{A_i\}$ be a collection of free factors and let $X_i$ be the set of vertices of $\C_n$ that do not project to $A_i$, i.e. $X_i  = X_{A_i}$. Let $M$ be the constant determined in Proposition \ref{BGITc}. Assume that

 \begin{enumerate}
\item $X_i$ and $X_{i+1}$ are disjoint in $\C_n$ and
\item $d_{Y_i}(x_{i-1}, x_{i+1}) > 2M$ for any $x_{i-1} \in X_{i-1}$ and $x_{i+1} \in X_{i+1}$. 
\end{enumerate}
Then the $X_i$ are pairwise disjoint and for any $x_j \in X_j$ and $x_{j+k} \in X_{j+k}$, any geodesic $[x_j,x_{j+k}]$ contains a vertex from $X_i$ for $j \le i \le j+k $.
\end{proposition}

\begin{proof}
The proof is adapted from \cite{Ma2}. There are two reasons for providing the details here. First, the argument is an illustration of how the general subfactor projections discussed in this note and the complex $\C_n$ are many ways analogous to subsurface projections and the curve complex. Second, there are several subtleties that make subfactor projections different; for example, there is no canonical ``boundary curve'' of $A$ contained in $X_A$.

The proposition is proven by induction on $k$; for $k=1$ there is nothing to prove. Let $x_j 
\in X_j$ and $x_{j+k} \in X_{j+k}$ be given and consider the geodesic $[x_j,x_{j+k}]$. Select any $x_{j+k-2} \in X_{j+k-2}$. We first show that there exists a geodesic $[x_j,x_{j+k-2}]$ that avoids vertices of $X_{j+k-1}$. To see this, start with a geodesic $[x_j, x_{j+k-2}]$ that contains a vertex $x_{j+k-1}$ of $X_{j+k-1}$ and decompose it as
$$[x_j, x_{j+k-2}]  = [x_j, x_{j+k-1}] \cup [x_{j+k-1}, x_{j+k-2}].$$
Suppose that we have chosen $x_{j+k-1}$ to be the first vertex of $X_{j+k-1}$ that appears along $[x_j, x_{j+k-2}]$ so that $[x_j, x_{j+k-1}]$ is disjoint from $X_{j+k-1}$ except at its last vertex. The induction hypotheses now implies that $[x_j, x_{j+k-1}]$ meets $X_{k+j-2}$ at a vertex $x_{k+j-2}'$ and we can write
$$[x_j, x_{j+k-2}]  = [x_j, x'_{j+k-2}] \cup [x'_{j+k-2}, x_{j+k-1}]  \cup [x_{j+k-1}, x_{j+k-2}]. $$

By assumption, these last two geodesics have length at least $1$ and since the diameter of each $X_i$ is less that $2$ we may replace the union of the last two geodesics with a geodesic $\{x_{j+k-2}', a_{j+k-2}, x_{j+k-2}\}$, where $a_{j+k-2}$ is a cyclic factor of $A_{j+k-2}$ whose projection to $A_{j+k-1}$ is nonempty. This is possible by Lemma \ref{cyclicsuffices}. Hence, we have produced a geodesic from $x_j$ to $x_{j+k-2}$ that avoids $X_{j+k-1}$.

Since $[x_j,x_{j+k-2}]$ avoids $X_{j+k-1}$, Proposition \ref{BGITc} implies that $d_{A_{j+k-1}}(x_j,x_{j+k-2}) \le  M$. Hence, 
\begin{eqnarray*}
d_{A_{j+k-1}}(x_j,x_{j+k}) &\ge& d_{A_{j+k-1}}(x_{j+k-2},x_{j+k}) -d_{A_{j+k-1}}(x_j,x_{j+k-2}) \\
&>& 2M -M \ge M.
\end{eqnarray*}
Another application of Proposition \ref{BGITc} gives that any geodesic $[x_j,x_{j+k}]$ must contain a vertex that misses $A_{j+k-1}$, hence there is a vertex $x_{j+k-1} \in X_{j+k-1}$ with $x_{j+k-1} \in [x_j,x_{j+k}]$. This implies that we may write $[x_j,x_{j+k}] = [x_j,x_{j+k-1}] \cup [x_{j+k-1}, x_{j+k}]$ and applying the induction hypothesis to $[x_j,x_{j+k-1}]$ we conclude that the geodesic $[x_j,x_{j+k}]$ contains a vertex from each $X_i$ for $j \le i \le j+k$. Also, if $X_j \cap X_{j+k}$ contained a vertex $x$ then the geodesic $[x,x]$ would have to intersect $X_{j+1}$, contradicting our hypothesis. This conclude the proof.
\end{proof}

The next theorem is similar to Proposition $3.3$ in \cite{Ma2}, where pseudo-Anosov mapping classes are constructed using the curve complex. Say that a collection of free factors $\{A_1,\ldots, A_n\}$ of $F_n$ \emph{fill} if for any free factor $C < F_n$, $\pi_{A_i}(C) \neq \emptyset$ for some $i$.  In other words, every free factor meets some factor in the collection.

\begin{theorem} \label{constructinggeneral}
Let $A$ and $B$ be rank $\ge 2$ free factors of $F_n$ that fill and let $f,g \in \Out(F_n)$ satisfy the following:
\begin{enumerate}
\item $f(A) = A$ and $f|_A \in \Out(A)$ has translation length $> 2M+4D$, and
\item $g(B) =B$ and $g|_B \in \Out(B)$ has translation length $> 2M+4D$.
\end{enumerate}
Then any nontrivial automorphism in the subgroup $\langle f,g \rangle \le \Out(F_n)$ that is not conjugate to a power of $f$ or $g$ is fully irreducible. Moreover, any finitely generated subgroup of $\langle f,g \rangle$ consisting entirely of such automorphisms has the property that any orbit map into $\FF_n$ is a quasi-isometric embedding.
\end{theorem}

Before beginning the proof we make the following remark: By \cite{BFhyp}, an outer automorphism has positive translation length in $\FF_n$ (or $\C_n$) if and only if it is fully irreducible. Hence, if $f$ and $g$ fix the free factors $A$ and $B$, respectively, and their restrictions are fully irreducible, then conditions $(1)$ and $(2)$ are satisfied after passing to a sufficiently high power. If there were to exists a uniform lower bound on the translation length of a fully irreducible automorphism in $\FF_n$, depending only on $n$, then such a power would be independent of $f$ and $g$. 

\begin{proof}

We sketch the proof as the details are similar to \cite{Ma2}. First, note that we have chosen translation lengths sufficient large so that any geodesic of $\C_n$ joining vertices of  $X_B$ and $fX_B$ must contain a vertex of $X_A$, and similarly any geodesic joining vertices of  $X_A$ and $gX_A$ must contain a vertex of $X_B$. To see this, note that since $A$ and $B$ fill, $X_A \cap X_B = \emptyset$. Also, if $b$ is a rank $1$ free factors of $B$ that meets $A$, which exists by Lemma \ref{cyclicsuffices},
$$\mathrm{diam}(\pi_A(X_B)) \le 2 \cdot \max\{d_A(b,\beta): \beta \in X_B \} \le 2D. $$
Hence, for any $\beta \in X_B$ and $\beta' \in fX_B$, let $a_{\beta} \in \pi_A(\beta)$ so that 
\begin{eqnarray*}
d_A(\beta,\beta') &\ge& d_A(\beta,f\beta) - d_A(f\beta,\beta') \\
                &\ge& d_A(a_{\beta},fa_{\beta}) - 2\cdot \mathrm{diam}(\pi_A(\beta)) - d_A(f\beta,\beta') \\
                &>& (2M+4D) -2D-2D \ge 2M > M.
\end{eqnarray*}
Then by the Proposition \ref{BGITc}, any geodesic from $\beta$ to $\beta'$ must contain a vertex that misses $A$, i.e. that is contained in $X_A$. The proof now proceeds by using Theorem \ref{progress} to show that elements not conjugate to powers of $f$ or $g$ act with positive translation length on $\C_n$ and are therefore fully irreducible.

For any $w \in \langle f,g \rangle$ in reduced form, write $w =s_1\ldots s_n$ where each $s_i$ is a \emph{syllable} of $w$, i.e. a maximal power of either $f$ or $g$. Suppose for simplicity that $s_1$ is a power of $f$ and $s_n$ is a power of $g$ (so in particular $n$ is even) and set $X_i = s_1 \ldots s_{i-1}X_A$ for $i$ odd and $X_i = s_1 \ldots s_{i-1}X_B$ for $i$ even. By naturally of the $\Out(F_n)$-action, these are precisely the sets of vertices of $\C_n$ that fail to project to the free factors $A_i = s_1 \ldots s_{i-1}A$ and $B_i = s_1 \ldots s_{i-1}B$, respectively. Using that fact that $X_A$ is fixed by $f$ and $X_B$ is fixed by $g$, it is quickly verified that the sets $X_i$ satisfy the conditions of Theorem \ref{progress} for $1 \le i \le n+1$. We conclude that for $\alpha \in X_A$ and $w \alpha \in wX_A = X_{n+1}$, the geodesic $[\alpha, w\alpha]$ contains at least $n+1$ vertices and so
$$d_{\C} (\alpha, w \alpha) \ge n =  |w|_s, $$
where $|\cdot|_s$ denotes the number of syllables. In general, one shows that either $d_{\C} (\alpha, w \alpha) \ge |w|_s$ or $d_{\C} (\beta, w \beta) \ge |w|_s$, where $\beta \in X_B$, depending on the first and last syllable of $w$.

To finish the proof, observe that any $w \in  \langle f,g \rangle$ that is not a conjugate to a power of $f$ or $g$ has a conjugate $w'$ with an even number of syllables and $w'$ has the property that $|w'^n|_s = n |w'|_s$. Hence, $w'$ has positive translation length in $\C_n$, as does its conjugate $w$. This shows that $w$ is fully irreducible.

 The statement about quasi-isometric orbit maps follows as in \cite{Ma2}.
\end{proof}

We conclude with the remark that Theorem \ref{constructinggeneral} can be generalized to free groups of higher rank as well as the right-angled Artin subgroups of $\Out(F_n)$ constructed in \cite{Tayl1}. This will be the subject of future work.

\bibliographystyle{amsalpha}
\bibliography{subfactorprojections.bbl}

\providecommand{\bysame}{\leavevmode\hbox to3em{\hrulefill}\thinspace}
\providecommand{\MR}{\relax\ifhmode\unskip\space\fi MR }
\providecommand{\MRhref}[2]{%
  \href{http://www.ams.org/mathscinet-getitem?mr=#1}{#2}
}
\providecommand{\href}[2]{#2}
\begin{thebibliography}{MM00}

\bibitem[Beh06]{Be}
Jason Behrstock, \emph{Asymptotic geometry of the mapping class group and
  {T}eichm\"uller space}, Geom. Topol. \textbf{10} (2006), 1523--1578.

\bibitem[Bes10]{Bbers}
Mladen Bestvina, \emph{A {B}ers-like proof of the existence of train tracks for
  free group automorphisms}, arXiv preprint arXiv:1001.0325 (2010).

\bibitem[BF11]{BFhyp}
Mladen Bestvina and Mark Feighn, \emph{Hyperbolicity of the complex of free
  factors}, arXiv preprint arXiv:1107.3308 (2011).

\bibitem[BF12]{BFproj}
\bysame, \emph{Subfactor projections}, arXiv preprint arXiv:1211.1730 (2012).

\bibitem[CP10]{CPtwisting}
Matt Clay and Alexandra Pettet, \emph{Twisting out fully irreducible
  automorphisms}, Geometric and Functional Analysis \textbf{20} (2010), no.~3,
  657--689.

\bibitem[CV86]{CVouter}
Marc Culler and Karen Vogtmann, \emph{Moduli of graphs and automorphisms of
  free groups}, Inventiones mathematicae \textbf{84} (1986), no.~1, 91--119.

\bibitem[FM11]{FMout}
Stefano Francaviglia and Armando Martino, \emph{Metric properties of outer
  space}, Publicacions Matem{\`a}tiques \textbf{55} (2011), no.~2, 433--473.

\bibitem[HH12]{HHhyp}
Arnaud Hilion and Camille Horbez, \emph{The hyperbolicity of the sphere complex
  via surgery paths}, arXiv preprint arXiv:1210.6183 (2012).

\bibitem[HM11]{HMsplit}
Michael Handel and Lee Mosher, \emph{The free splitting complex of a free group
  {I}: Hyperbolicity}, arXiv preprint arXiv:1111.1994 (2011).

\bibitem[HV98]{HVff}
Allen Hatcher and Karen Vogtmann, \emph{The complex of free factors of a free
  group}, Quarterly Journal of Mathematics \textbf{49} (1998), no.~196,
  459--468.

\bibitem[KIL08]{KLconv}
Richard~P Kent~IV and Christopher~J Leininger, \emph{Uniform convergence in the
  mapping class group}, Ergodic Theory and Dynamical Systems \textbf{28}
  (2008), no.~4, 1177--1196.

\bibitem[KL10]{KLping}
Ilya Kapovich and Martin Lustig, \emph{Ping-pong and outer space}, Journal of
  Topology and Analysis \textbf{2} (2010), no.~02, 173--201.

\bibitem[KR12]{KapRaf}
Ilya Kapovich and Kasra Rafi, \emph{On hyperbolicity of free splitting and free
  factor complexes}, arXiv preprint arXiv:1206.3626 (2012).

\bibitem[Man]{Ma2}
Johanna Mangahas, \emph{A recipe for short-word pseudo-{A}nosovs}, Proceedings
  of the American Mathematical Society (2010).

\bibitem[MM00]{MM2}
Howard~A. Masur and Yair~N. Minsky, \emph{Geometry of the complex of curves.
  {II}. {H}ierarchical structure}, Geom. Funct. Anal. \textbf{10} (2000),
  no.~4, 902--974.

\bibitem[SS12]{SS}
Lucas Sabalka and Dmytro Savchuk, \emph{Submanifold projection}, arXiv preprint
  arXiv:1211.3111 (2012).

\bibitem[Tay13]{Tayl1}
Samuel~J Taylor, \emph{Right-angled {A}rtin groups and {O}ut ({F}$_n$) {I}:
  quasi-isometric embeddings}, arXiv preprint arXiv:1303.6889 (2013).

\end{thebibliography}
\end{document}